\documentclass[12pt,reqno]{amsart}
\usepackage[dvips]{graphics}
\usepackage{amssymb}
\usepackage{dsfont}

\numberwithin{equation}{section}

\newtheorem{thm}{Theorem}[section]
\newtheorem*{thm*}{Theorem}
\newtheorem*{thmmain*}{MAIN THEOREM}
\newtheorem{lem}[thm]{Lemma}
\newtheorem{cor}[thm]{Corollary}
\newtheorem{prop}[thm]{Proposition}
\newtheorem*{prop*}{Proposition}
 
\theoremstyle{definition}
\newtheorem{defn}{Definition}[section]
\theoremstyle{remark}
\newtheorem{rem}{Remark}[section]
\newtheorem{ex}[rem]{Example}
\newtheorem{quest}[rem]{Question}

\newcommand{\tref}[1]{Theorem~\ref{#1}}
\newcommand{\cref}[1]{Corollary~\ref{#1}}
\newcommand{\pref}[1]{Proposition~\ref{#1}}

\newcommand{\lref}[1]{Lemma~\ref{#1}}

\def\R{\mathds{R}}
\def\dim{\mathop{\text{dim}}}

\def\codim{\mathop{\text{codim}}}
\def\Id{\mathop{\text{Id}}}

\begin{document}
\title{Geometric resolution of singular Riemannian foliations}

\author{Alexander Lytchak}
\address{A. Lytchak, Mathematisches Institut, Universit\"at Bonn,
Beringstr. 1, 53115 Bonn, Germany}
\email{lytchak\@@math.uni-bonn.de}



\begin{abstract}
We prove that an isometric action  of a Lie group on a Riemannian manifold  admits a resolution preserving 
the transverse geometry if and only if the action is infinitesimally polar. We  provide applications concerning topological simplicity of several classes of isometric actions, including polar and variationally complete ones. All results are proven in the more general case of singular
Riemannian foliations.
\end{abstract}

\thanks{The author was supported in part by the SFB 611 
{\it Singul\"are Ph\"anomene und Skalierung in mathematischen Modellen.} and by the MPI for mathematics in Bonn}

\maketitle
\renewcommand{\theequation}{\arabic{section}.\arabic{equation}}

\pagenumbering{arabic}

\section{Introduction}
For an isometric action of a Lie group $G$ on a Riemannian manifold $M$
the presence of singular orbits is the main source of difficulties to 
understand the geometric and topological properties of the action. It 
 seems natural to look for   some procedure resolving the singularities,
i.e., some way to pass from $M$ to some other $G$-manifold $\hat M$ with
only regular orbits, related to $M$ in some canonical way.  For the choice of 
the procedure it is crucial, what kind of information one would like 
to preserve by this resolution. If one only would like to let the 
regular part of the action unchanged, then there is a canonical procedure
resolving an arbitrary action. One starts with the most singular stratum, 
replaces it by the projectivized normal bundle and proceeds inductively.
The reader is referred,  for instance, to \cite{wasserman} or to
\cite{molinodes} for this topological approach.
The disadvantage of this method is that many crucial geometric 
and topological properties of the action  are ``concentrated'' in the singular
locus and in the transverse geometry and  cannot be traced by this 
procedure. 

 In geometry it seems natural to consider the quotient 
$M/G$ with the induced metric as the essence of the action. Thinking of the 
action as  of a (singular) foliation, one considers the transverse geometry 
as the  most important object.  Therefore it seems natural to look
 only for such resolutions $\hat M$ with a $G$-equivariant surjective map 
$f:\hat M \to M$ such that the induced map $f: \hat M /G \to M/G$ is an 
isometry (some partial 
 resolutions of this type have already been considered, for 
instance in  \cite{grovesearle}).
 The main technical result of this paper (\tref{mainthm})
states that such a resolution exists if and only if all isotropy representations
of the action are polar.  Many natural classes of actions, for instance
polar ones, variationally complete ones or actions of cohomogeneity at most two
satisfy  this property of being infinitesimally polar. Moreover, if the action
is infinitesimally polar there is a canonical resolution that inherits
many properties of the original action.  This provides a way to 
reduce the study of some topological and geometric properties of actions
 to the case of regular actions, where they can be easily established,
see the subsequent results in the introduction.

 It turns out that the action itself does not play a role in our 
considerations, but only the decomposition of the manifold into orbits, i.e.,
a {\it singular Riemannian foliation}.  
 We refer the reader to \cite{Molino} or to the preliminaries in  
 Section \ref{prelim} for  basics about singular Riemannian foliations.
Readers only interested in the special case of group actions may just consider
all singular Riemannian foliations as orbit decompositions of an isometric 
group action. We also would like to mention \cite{wiesend},
where the ideas of this paper are elaborated 
and simplified in the  case of isometric group actions.

\begin{defn}
Let $\mathcal F$ be a singular Riemannian foliation on a
  Riemannian manifold $M$. A geometric resolution of $(M,\mathcal F)$ is a 
smooth surjective map $F:\hat M\to M$ from a smooth Riemannian manifold 
$\hat M$ with a regular Riemannian foliation $\hat {\mathcal F}$ such that 
the
following holds true. For all smooth curves $\gamma$ in $\hat M$ the transverse
lengths of $\gamma$ with respect to $\hat {\mathcal F}$ and of $F (\gamma )$ 
with respect to $\mathcal F$ coincide.   
\end{defn}

Here the transverse length  is defined as usual in the theory of foliations
as the length of the projection to local quotients 
(Subsection \ref{transsubsec}).  The last requirement in the 
definition above  
means  that  $F$ sends leaves of $\hat {\mathcal F}$ to leaves of
 $\mathcal F$ and induces a length-preserving map between the quotients
$F:\hat M /\hat {\mathcal F} \to M/\mathcal F$, see Section \ref{onlyif}. 
Considering the quotient space $M/\mathcal F$ with its local metric
structure as the essence of the singular Riemannian foliation
$(M,\mathcal F)$, the above definition becomes the most natural one.

 Our main result reads as follows:

\begin{thm} \label{mainthm}
 Let $M$ be a Riemannian manifold and let $\mathcal F$ be a singular
Riemannian foliation on $M$. Then $(M,\mathcal F)$ has a geometric
resolution if and only if $\mathcal F$ is infinitesimally polar. If 
$\mathcal F$ is infinitesimally polar then there is a canonical resolution
  $F:\hat M  \to M$ with the following properties. The resolution $\hat M$
is of the same dimension as $M$ and the map $F$ induces
a bijection between the spaces of leaves.  Moreover, $F$ is a diffeomorphism,
when restricted to the preimage of the set of regular points of 
$(M,\mathcal F)$. The map $F$ is proper and $1$-Lipschitz. In particular,
the resolution $\hat M$
 is compact or complete if $M$
has the corresponding property.  The isometry group $\Gamma$ of 
$(M,\mathcal F)$ acts by isometries on $(\hat M, \hat {\mathcal F})$ and
the map $F:\hat M\to M$ is $\Gamma$-equivariant. If $\mathcal F$
is given by the orbits of a group $G$ of isometries of $M$ then
$G$ acts by isometries on $\hat M$, and $ \hat {\mathcal F}$
is given by the orbits of $G$. If $M$  is complete
 then the singular Riemannian foliation $\mathcal F$ has no 
horizontal conjugate 
points if and only
if $ \hat {\mathcal F}$  has no horizontal conjugate points.  If $M$
is complete then the singular
Riemannian foliation  
$\mathcal F$ is polar  if and only if   $ \hat {\mathcal F}$ is polar.
\end{thm}

The  infinitesimal polarity of $\mathcal F$ means that locally
the singular Riemannian foliation $\mathcal F$ is diffeomorphic to
an isoparametric  singular Riemannian foliation on a Euclidean
space (Subsection \ref{infpolsubsec}). 
For  polar singular Riemannian foliations we refer to Subsection 
\ref{polarsubsec}
(cf. \cite{Terng}, \cite{Boualem}, \cite{Alexandrino}, \cite{Alexandrino1})
 and for singular Riemannian foliations without horizontal  conjugate 
points we refer to  \cite{BottS}, \cite{LT}
\cite{expl}.

Before we are going to comment on this theorem and related results, we state 
some consequences that motivated our study of geometric resolutions. Recall
that a  (regular) Riemannian foliation $\mathcal F$ on a Riemannian manifold
$M$ is called {\it simple} if it is given by the fibers of a Riemannian submersion.
If $M$ is complete (or, more generally, if $\mathcal F$ is full, see 
Section \ref{toposec}) then $\mathcal F$ is simple if and only if all
leaves of $\mathcal F$ are closed and have no holonomy (\cite{hermann}).
The next result generalizes \cite{blum-hebda} and \cite{Hebda}, Theorem 2 to
the realm of  singular Riemannian foliations.

\begin{thm} \label{proper}
Let $M$ be a complete, simply connected Riemannian manifold, and let 
$\mathcal F$ be a singular Riemannian foliation on $M$.  If $\mathcal F$
is polar, or if $\mathcal F$ has no horizontal conjugate points 
then the leaves of $\mathcal F$ are closed. Moreover, the restriction
of $\mathcal F$ to the regular part of $M$ is a simple foliation.
\end{thm}

In \cite{Terng} it is shown that isoparametric foliations on simply connected 
 spaces of constant curvature have closed leaves and that there are no 
{\it exceptional leaves}, i.e., that all regular leaves have trivial holonomy. 
In  \cite{tobenphd} it is shown that if $\mathcal F$ is a polar singular
Riemannian foliation on  a simply connected 
symmetric space $M$ then properness of all leaves implies  
vanishing of holonomy of regular  leaves. Finally, in \cite{altoben} the same
result was shown for an arbitrary complete, simply connected space $M$.  
Thus, in the case of polar singular Riemannian foliations only the closedness
of $\mathcal F$ is new.

 Since a connected group of isometries of a Riemannian manifold is closed  
if and only its orbits are closed, \tref{proper} reads in the case of
group actions as follows:

\begin{cor}
Let $M$ be a complete, simply connected manifold and let a connected group
$G$ act by isometries of $M$. If the action is polar or variationally
complete then  the image of $G$ in the isometry group of $M$ is closed
and there are no exceptional orbits of the action.
\end{cor}

 From \tref{proper} and \cite{expl}, Theorem 1.7 we immediately get
a complete description of singular Riemannian  foliations without 
horizontal conjugate points in terms of their quotient spaces. Since
 singular Riemannian  foliations without  horizontal conjugate points 
generalize  the concept of variationally complete actions introduced in 
\cite{Bott} and \cite{BottS} and investigated in  \cite{Conlon},\cite{GT},
\cite{DO} and \cite{LT},
the next result also gives a description of variationally complete actions
in terms of the quotient spaces.  Since complete non-negatively
curved Riemannian orbifolds without conjugate points are flat,
the next result generalizes the main results of \cite{DO}, \cite{GT} and
\cite{LT}.

\begin{cor} \label{coro}
Let $M$ be a complete Riemannian manifold and let $\mathcal F$ be a singular
Riemannian foliation. Then $\mathcal F$ does not have horizontal conjugate
points if and only if the lift $\tilde {\mathcal F}$ of $\mathcal F$ to the
universal covering $\tilde M$ of $M$ is closed and the quotient 
$\tilde M/\tilde {\mathcal F}$ is a Riemannian orbifold without conjugate 
points. 
\end{cor}

To deduce \tref{proper} from \tref{mainthm} we proceed as follows. If 
$\mathcal F$ is polar then $\mathcal F$ is also infinitesimally polar.
If $\mathcal F$ has no horizontal conjugate points then it is
infinitesimally polar as well, due to \cite{expl}, Theorem 1.7. Thus
we may apply \tref{mainthm} and obtain a regular Riemannian foliation
$\hat {\mathcal F}$ on a complete Riemannian manifold $\hat M$ that
is polar or has no horizontal conjugate points. In the first case  we  apply
\cite{blum-hebda} and  deduce that the lift
of $\hat {\mathcal F}$ to the universal covering of $\hat M$ is
a simple Riemannian foliation. In the second case, the 
leaves of the regular Riemannian foliation $\hat {\mathcal F}$ on the
complete Riemannian manifold $\hat M$ have no focal points and
the proof of \cite{Hebda}, Theorem 2 shows that   the lift
of $\hat {\mathcal F}$ to the universal covering of $\hat M$ is again 
a simple Riemannian foliation.  But $(\hat M,\hat {\mathcal F})$
coincides with $(M,\mathcal F)$ on the regular part $M_0$ of $M$. 
Therefore, the restriction of $\mathcal F$ to $M_0$ becomes simple,
when lifted to the universal covering $\tilde M_0$ of $M_0$.
Thus, \tref{proper}  follows from the next general topological observation
whose proof   will be given in Section \ref{toposec}. The proof of this result
is implicitly  contained in \cite{Molino}, p.213-214 
(see also \cite{molinors}).

\begin{thm} \label{topothm}
Let $M$ be a complete, simply connected Riemannian manifold and let 
$\mathcal F$ be a singular
Riemannian foliation on $M$. If the restriction of $\mathcal F$ to the
regular part $M_0$ becomes simple, when lifted to the universal covering 
$\tilde M_0$ of $M_0$,  
then the restriction of $\mathcal F$ to $M_0$ is a simple
foliation. 
\end{thm}

 In Section \ref{toposec} we will  discuss a more general version
 of the theorem above.  We also will derive   some further consequences of \tref{topothm}
 concerning the general structure of infinitesimally polar foliations
 with closed leaves  on simply connected manifolds. These  results are independent of our main
 \tref{mainthm}, but are related to \tref{proper} and therefore included here.
 To state these results we will need some notations.
 
  For a
  Riemannian orbifold  $B$,  we denote by $\partial B$ 
 the union of all closures of all singular strata of $B$ that have codimension $1$ in $B$  and call it 
  the {\it boundary} of  $B$
 (This coincides with the boundary in the sense of Alexandrov geometry). 
 
 \begin{thm} \label{boundary}
 Let $\mathcal F$ be a closed infinitesimally polar singular Riemannian foliation on a complete
 manifold $M$ with quotient orbifold $B$. Then all singular leaves of $\mathcal F$ are contained in the
boundary   $\partial B$. If $M$
 is simply  connected then the converse is also true, i.e., $\partial B$ is the set of all singular leaves.
 In particular, for simply connected $M$, the quotient $B$ has no boundary if and only if $\mathcal F$ is a regular 
 foliation.  
 \end{thm}

 For foliations of codimension $2$ we  will deduce from the last theorem a result  generalizing a known 
 statement about    compact transformation groups (Theorem 8.6 in Chapter IV of \cite{Bredon}):

\begin{cor} \label{cor2d}
Let $M$ be a complete simply connected Riemannian manifold and let $\mathcal F$ be a closed singular Riemannian foliation
with a quotient $B=M/\mathcal F$ of dimension $2$. Then either the foliation is regular
or there are no exceptional leaves. 
 \end{cor}

  For further investigations of exceptional orbits we need another definition.
A {\it Coxeter orbifold} 
(cf. \cite{coxeter}) is  Riemannian orbifold locally diffeomorphic to {\it Weyl chambers},
i.e., to  quotients  of the Euclidean space by  finite  Euclidean Coxeter groups.
Note that in a Coxeter orbifold
each non-manifold  point is contained in the boundary. In  dimension $2$ the converse
holds as well, i.e., a two-dimensional  orbifold is a Coxeter orbifold if it does not have isolated 
singularities. In particular, a Coxeter orbifold does not have to be a {\it good} orbifold, as it was
claimed in  \cite{coxeter} and cited in the previous version of this paper (a disc with an additional conical
singularity on the boundary is a counterexample, cf. Remark \ref{wrong}).

 Now we can state:

\begin{thm} \label{goodorb}
Let $M$ be a complete, simply connected Riemannian manifold and let
$\mathcal F$ be a closed infinitesimally polar singular Riemannian foliation on $M$ with quotient
$B=M/\mathcal F$. Then the following are equivalent:

\begin{enumerate}
\item There are no exceptional leaves;
\item The regular part $B_0:= M_0 /\mathcal F$ is a good orbifold;
\item The quotient $B$ is a Coxeter orbifold;
\item All non-manifold points of the orbifold $B$ are contained in the boundary $\partial B$.
\end{enumerate} 
\end{thm}

\begin{ex}
 Closed singular Riemannian foliations that are polar or have no horizontal
conjugate points have good Riemannian orbifolds as quotients (thus $B_0$ is good as well). In the case
of closed polar singular Riemannian foliations on simply connected manifolds
it was shown in \cite{altoben}, that the quotients are Coxeter orbifolds.
\end{ex}

\begin{ex}
If  the singular Riemannian foliation $\mathcal F$ is induced
by the action of a connected group $K$ of isometries, the equivalent conditions of \tref{goodorb}
are also equivalent to the following one: For all $x$ in $M$ the action of the isotropy group
$K_x$ on the horizontal space $H_x$ has connected fibers.  In fact, the sufficiency is clear. Assume on
the other hand that there are no exceptional orbits. Then the  finite group $K_x /K_x ^0$
acts on the quotient $H_x /K_x ^0$ which is a Weyl chamber. The  set of its  regular point is contractible,
thus if the action of $K_x /K_x ^0$ is non-trivial there are elements of $K_x$ that fix some but not all points
in  $H_x /K_x ^0$. But such points correspond to exceptional orbits. See also the proof of  \tref{goodorb}, where the same argument is used.
\end{ex}

 In view of \tref{goodorb} it seems natural to ask the following
\begin{quest}
What simply connected Coxeter orbifolds $B$ can be represented as 
quotient spaces $B=M/\mathcal F$ for some singular Riemannian foliation
$\mathcal F$ on some  simply connected Riemannian manifold $M$.
\end{quest}

\begin{rem} \label{wrong}
Note, that if under the assumptions of \tref{goodorb}, the quotient  $B$ is a good orbifold, then $B_0$ is a 
good orbifold  as well, thus $B$ is a Coxeter orbifold. On the other hand, 
using \cite{haefligerfund} and the arguments of \cite{coxeter}, it is not difficult to deduce, that a 
Coxeter orbifold, that is simply connected as a topological space, is a  good  orbifold 
if and only if the following two conditions are fulfilled:
A wall (a stratum of codimension $1$) intersects a small tube around any stratum of codimension $2$ in a connected set.
If the closures of two walls intersect at different connected components, then the intersection angles at these
components do not depend on the component.  Thus,  it is not to difficult to decide, when the quotient 
$B$ as in \tref{goodorb} is a good orbifold.
\end{rem}

 The proof of \tref{mainthm} is provided in Section \ref{dessec} and Section
\ref{onlyif} along
the following lines.  For an infinitesimally polar $\mathcal F$
on a Riemannian manifold $M$ one uses the ideas of \cite{Boualem} and
\cite{tobenphd}  and defines  the resolution
$\hat M$ to be the subset of the Grassmannian bundle $Gr_k (M)$
consisting of all infinitesimal horizontal sections of $\mathcal F$.
In the polar case the result is  contained in \cite{Boualem} and 
\cite{tobenphd}. In the general case one follows an idea from
\cite{expl} and uses  transformation  relating
horizontal geometry of different Riemannian metrics 
adapted to a given foliation  to reduce the question to the polar case.

\begin{rem}
 The proof shows (and is  based on) the fact that the resolution
$(\hat M,\hat {\mathcal F})$ considered as a  foliation on
a  manifold (disregarding the Riemannian metric on $\hat M$)  does not
depend on the Riemannian metric adapted to the singular Riemannian foliation
$\mathcal F$ on $M$. 
\end{rem}

 To see that a singular Riemannian foliation  
$\mathcal F$ with a metric resolution $\hat {\mathcal F}$ is infinitesimally
polar one observes that in a regular Riemannian foliation transversal
sectional curvatures remain bounded on compact subsets. Now, one uses 
the transverse equivalence of $\mathcal F$  and $\hat {\mathcal F}$ and 
deduces  from \cite{expl}, Theorem 1.4 that this property characterize 
infinitesimally polar singular Riemannian foliations. 
 This already proves the claim in the case of a
compact resolution $\hat M$. In the general case one needs to be more
careful and to extend  some  results  from \cite{expl} slightly (\lref{onegeod}).

 We would like to mention that Sections \ref{onlyif}, \ref{dessec} and 
\ref{toposec}  do not depend on each other.
Thus, reader only interested in \tref{topothm} and subsequent results may directly proceed to
Section \ref{toposec} and reader only interested in the (more important) 
if part of \tref{mainthm} may skip Section \ref{onlyif}.

\subsection{Acknowledgments} I am grateful to Gudlaugur Thorbergsson
for helpful conversations and to Stephan Wiesendorf for useful comments on a 
previous version of the paper. I  thank the referee for the careful reading
of the manuscript.

\section{Preliminaries} \label{prelim}
\subsection{Singular Riemannian foliations}
 A {\it transnormal system} $\mathcal F$ on a Riemannian
manifold $M$ is a decomposition of $M$ into smooth injectively immersed 
connected submanifolds, called leaves,  such that geodesics 
emanating perpendicularly to one leaf stay perpendicularly to all
leaves. A transnormal system $\mathcal F$ is called a {\it singular Riemannian
foliation} if  there are smooth vector fields $X_i$
on $M$ such that for each point $p\in M$ the tangent space
$T_p L (p)$ of the leaf $L(p)$ through $p$ is given
as the span of the vectors $X_i (p) \in T_p M$. We refer to
\cite{Molino} and \cite{Wilk} for more on singular Riemannian foliations.
  Examples
of singular Riemannian foliations are (regular) 
Riemannian foliations and  the orbit decomposition of an isometric group
action.  

\subsection{Stratification}
 Let $\mathcal F$  be a singular Riemannian foliation  on the Riemannian
manifold $M$. The {\it dimension of $\mathcal F$},  $ \dim  (\mathcal F)$,
is the  maximal  dimension of  its leaves. The  {\it codimension of $\mathcal F$},
 $\codim  (\mathcal F, M)$,  is defined by $\dim (M) -\dim (\mathcal F )$.
For $s\leq \dim (\mathcal F)$, denote by $\Sigma _s$
the subset of all points $x\in M$  with $\dim (L(x))=s$. Then
$\Sigma _s$ is an embedded submanifold of $M$  and the restriction of
 $\mathcal F$ to $\Sigma _s$ is a Riemannian foliation. For a point
$x\in M$, we denote by $\Sigma ^x$ the connected component of $\Sigma _s$
through $x$, where $s =\dim (L(x))$.   We call the decomposition of $M$
into the manifolds $\Sigma ^x$ the {\it canonical stratification} of $M$.

 The subset $\Sigma _{\dim (\mathcal F)}$ is open, dense and connected
in $M$. It is the {\it regular stratum}  $M$. It will be denoted by $M_0$
and will also be called the set or regular points of $M$.
 All other  strata  $\Sigma ^x$, called {\it singular strata},
 have codimension at least $2$ in $M$.   For any singular stratum $\Sigma$,
we have $\codim (\mathcal F, \Sigma ) < \codim (\mathcal F, M)$.

\subsection{Infinitesimal singular Riemannian foliations} \label{infinitsing}
 Let $M$ be a Riemannian manifold and let $\mathcal F$ be a singular Riemannian
foliation on $M$. Let $x\in M$ be a point.  Then there is a well defined 
singular Riemannian foliation $T_x \mathcal F$ on the Euclidean space
$(T_x M,g_x)$ with the following properties:

\begin{enumerate}
 
\item There is a neighborhood $O$ of $x$  and a diffeomorphic embedding 
$\phi :O \to T_x M$, with $D_x \phi =Id$ and 
$\phi ^{\ast} (T_x \mathcal F) = \mathcal F | _{O}$.

\item $T_x \mathcal F$ is homogeneous, i.e., for each non-zero real number
$\lambda$, the multiplication by $\lambda $ on $T_x M$ preserves 
$T_x \mathcal F$.

\item  The singular foliation $T_x \mathcal F$ on the tangent
space $T_x M$ does not depend on the Riemannian metric adapted to $\mathcal F$.
\end{enumerate}

 The singular Riemannian foliation $T_x \mathcal F$ on the
tangent space $T_x M$ will be called the {\it infinitesimal 
singular Riemannian foliation  of $\mathcal F$ at the point $x$}.

\subsection{Horizontal sections}  \label{polarsubsec}
We refer to \cite{Boualem}, \cite{Alexandrino}, \cite{Alexandrino1}
 for more on polar singular Riemannian foliations.
Let $\mathcal F$ be a singular Riemannian foliation on a Riemannian manifold
$M$. A global (local) horizontal section through $x$ is a smooth immersed 
submanifold $x\in N \subset M$ that intersects all leaves of $\mathcal F$ 
(all leaves in a neighborhood of $x$), such that all intersections are
orthogonal. $\mathcal F$  is called polar (locally polar)
if there are (local) global horizontal
sections through every point $x\in M$. 
Each  local section $N$ of a singular Riemannian foliation
is totally geodesic. Moreover, for each $x\in N$, $T_x N \subset T_x M$
is a horizontal section of the infinitesimal singular Riemannian foliation
$T_x \mathcal F$.  On the other hand, if $\mathcal F$ is locally polar then  
each horizontal section $V\subset T_x M$ of the infinitesimal singular 
Riemannian foliation $T_x \mathcal F$ is the tangent space to a local
horizontal section of $\mathcal F$.

 Recall, that a singular Riemannian foliation $\mathcal F$ is locally polar
if and only if the restriction of $\mathcal F$ to the regular part 
$M_0$  has integrable horizontal distribution (\cite{Alexandrino1}). Moreover,
a locally polar singular Riemannian foliation on a complete Riemannian
manifold is polar.

\subsection{Infinitesimal polarity} \label{infpolsubsec}
The singular Riemannian foliation $\mathcal F$ is called infinitesimally
polar at the point $x\in M$ if the infinitesimal singular Riemannian
foliation $T_x \mathcal F$ is polar. We say that $\mathcal F$
is infinitesimally polar if it is infinitesimally polar at all
points. In \cite{expl} it is shown
that $\mathcal F$ is infinitesimally polar at the point $x$ if and only
if for all sequences $x_i$ of regular points converging to $x$,
the supremum $\bar \kappa (x_i)$  of the sectional curvatures 
at projections of $x_i$ to local quotients remain  bounded away from 
infinity. Another equivalent condition derived in \cite{expl}, is that
$\mathcal F$ is locally closed at $x$ and that  local quotients at $x$ are 
smooth Riemannian  orbifolds.

\subsection{Transverse length} \label{transsubsec}
Let $M$ be a Riemannian manifold and let $\mathcal F$ be a singular Riemannian
foliation on $M$.  For $x$ in $M$,  we denote by 
$V_x$ the tangent space to the 
leaf $V_x=T_x L(x)$ and call it the {\it vertical space at $x$}. The 
orthogonal complement of $V_x$ will be denoted by $H_x$  (or by $H_x (g)$,
if we want to specify the Riemannian metric $g$). This subspace $H_x$
will be called the {\it horizontal subspace at $x$}. 
 By $P_x:T_x \to H_x$
we denote the orthogonal projection.  The spaces $H_x$ vary semi-continuously.
Therefore, for each smooth curve $\gamma$ in $M$, the value
$L_{hor} (\gamma ):= \int |P_{\gamma (t)} (\gamma '(t))| dt$ is well defined.
We call this quantity the {\it transversal length of $\gamma$}.   If
$B=M/\mathcal F$ is a Hausdorff metric space then $L_{hor} (\gamma )$
is   the length of the projection of $\gamma$ to $B$. 
Note that a smooth curve has transversal length  zero if and only it
is completely contained in one leaf.

\section{The only if part} \label{onlyif}
We are going to prove the only if part of the first statement of
 \tref{mainthm}  in this section. Thus, let $\hat {\mathcal F}$ be
a regular Riemannian foliation on a Riemannian manifold $\hat M$, let
$\mathcal F$ be a singular Riemannian foliation on a Riemannian manifold $M$
and let $F:\hat M \to M$ be a geometric resolution. We are going to analyze
$F$ and to prove that $\mathcal F$ is infinitesimally polar. The proof
in the case of compact $\hat M$  was explained in the introduction. In 
the general case,  we will give a proof along the same lines, but 
the proof becomes  technically more involved.

 First of all, $F$ sends curves of zero transversal length to curves
of zero transversal length, therefore $F$ sends leaves into leaves, i.e.,
$F(\hat L (x))\subset L(F(x))$ for all $x\in \hat M$.   

For each open subset $O$  of $M$ the restriction $F:F^{-1} (O) \to O$
is again a geometric resolution. As usual, let $M_0$ denote  the set
of regular points of $M$ and set $\tilde M:= F^{-1} (M_0)$. Since the
restriction of $\mathcal F$ to $M_0$ is a {\it regular} Riemannian foliation,
we deduce from continuity reasons, that for all $x\in \tilde M$ the
map $G_x:=  P_{F(x)}\circ  D_x F: H_x \to H_{F(x)}$ is an isometric
embedding. Here, the horizontal subspaces $H$ and the projections $P$
are defined as in Subsection  \ref{transsubsec}.   

 On the other hand, $F$ is smooth and surjective. By Sard's theorem there is
at least one point $x\in \tilde M$ such that 
$D_x F :T_x \tilde M\to T_{F(x)} M$ is surjective. Since $D_x F$ sends 
$T_x (L(x))$ to a subspace of $T_{F(x)} (L(F(x)))$ we deduce that 
the map $G_x :H_x \to H_{F(x)}$ must be surjective at such points.
Therefore, $\dim (H_x) =\dim (H_{F(x)})$. Hence, 
$\codim (M, \mathcal F) = \codim (\hat M , \hat {\mathcal F})$. Moreover,
for each $x\in \tilde M$, the map $G_x :H_x \to H_{F(x)}$ is an isometry.

 Thus, for each point $x\in \tilde M$, we find a small neighborhood 
$O$ of $x$ such that $\hat {\mathcal F}$ on $O$ is given by a Riemannian
submersion $s_1:O\to B_1$, such that $\mathcal F$ on $F(O)$ is given by a 
Riemannian submersion $s_2 :F(O) \to B_2$, and such that $F$ induces
an isometry $\bar F :B_1 \to B_2$ between the local quotients.

 This finishes the analysis of $F$ on $\tilde M$. The picture over the 
singular points is more complicated. We start our discussion of the
singular part with the following easy observation.

\begin{lem} \label{curvbound}
 Let $\gamma _1 :[0,a] \to \hat M$ and $\gamma _2:[0,a] \to M$
be horizontal geodesics with $\gamma _2 ((0,a]) \subset M_0$. If 
$F(\gamma _1 (t)) \subset L (\gamma _2 (t))$, for all $t$, then
the sectional curvatures in local quotients at $L (\gamma _2 (t))$, 
$t\in (0,a]$, are uniformly bounded.  
\end{lem}

\begin{proof}
 From the discussion above we know that the sectional curvatures in 
local quotients at $L(\gamma _1 (t))$ and $\hat L (\gamma _2 (t))$ coincide
for all $t\in (0,a]$. Since $[0,a]$ is compact and $\hat F$ is a {\it regular}
Riemannian foliation, the sectional curvatures in local quotients at 
$\hat L (\gamma _2 (t))$ are uniformly bounded. 
\end{proof}

 The idea is now to find such curves starting at all points 
and to deduce  infinitesimal polarity from this existence.

\begin{lem}
 The open subset $\tilde M$ is dense in $\hat M$.
\end{lem}

\begin{proof}
Assume the contrary and choose an open subset $O$ of 
$\hat M\setminus \tilde M$.  By making $O$ smaller we may assume
that $F(O)$ is contained in a singular 
stratum $\Sigma $ of $M$. Now, the restriction
of $\mathcal F$ to $\Sigma $ is again a {\it regular} Riemannian foliation.
Thus, for each $x\in O$, we obtain by continuity
that $D_x F $ maps $H_x$ injectively onto the subspace $D_x F (H_x )$
that intersects $T_{F(x)} (L(F(x))$ only in $\{ 0 \}$. Thus we deduce
$$\codim (\hat M, \hat {\mathcal F}) =\dim (H_x) \leq 
\codim (\Sigma, \mathcal F) < \codim (M, \mathcal F)$$
since  $\Sigma$ is a singular stratum. This contradicts
the previously obtained equality 
$\codim (\hat M, \hat {\mathcal F})=  \codim (M, \mathcal F)$.
\end{proof}

Now we can prove:
\begin{lem} \label{manycurves}
For each $x\in M$, there are horizontal geodesics 
$\gamma _1 :[0,a] \to \hat M$ and $\gamma _2 :[0,a] \to M$ such that
$\gamma _2 (0)=x$, $\gamma _2 ((0,a] \subset M_0$ and  
$F(\gamma _1 (t)) \subset L (\gamma _2 (t))$, for all $t$.
\end{lem}

\begin{proof}
 Choose a distinguished tubular neighborhood $U$ at $x$ and a preimage
$y$ of $x$ in $\hat M$. Make the diameter $\epsilon$ of $U$ so small
that all geodesics starting in the $\epsilon$-neighborhood $O$ of
$y$ are defined at least for the time $\epsilon$.  Take a point
$z\in \tilde M\cap O$ with $\bar z= F(z)\in U$. Let $\bar x$
be the projection of $\bar z$ onto the leaf of $\mathcal F$ through 
$x$ in  $U$. Then $\bar x$ is the only possibly non-regular point
on the geodesic $\gamma _3 =\bar z \bar x$. Consider the horizontal
geodesic $\gamma _1$ in $\hat M$ starting at $z$ in the direction
$h$ with $G_z (h)= \gamma _3 '$. From the understanding
of $F$ on $\tilde M$, we deduce that $F( \gamma _1 (t))$ is contained
in $L (\gamma _3 (t))$ for all $t\in [0, d(\bar z, \bar x)]$. Now, replacing
$\gamma _3$ through a horizontal geodesic starting
in a point on $L(\bar z)$ and ending in $x$,  we obtain
a horizontal geodesic $\gamma _2$ ending in $x$ 
with $F(\gamma _1 (t))\in L(\gamma _2 (t))$. It remains to reverse the
orientations of $\gamma _1$ and $\gamma _2$.
\end{proof}

 Now the proof of the infinitesimal polarity of $\mathcal F$
is finished by combining \lref{curvbound}, \lref{manycurves} and the following
lemma, that  we consider to be of independent interest.

\begin{lem} \label{onegeod}
Let $\mathcal F$ be a singular Riemannian foliation on a Riemannian manifold
$M$. Let $x\in M$ be a point. Let 
$\gamma :[0,\epsilon ] \to M$ be a horizontal geodesic
starting at $x$, such that 
$\gamma ((0,\epsilon ])$ is contained in the set of regular points $M_0$.
If all sectional curvatures in local quotients are uniformly 
bounded along $\gamma (0,\epsilon ]$ then $\mathcal F$ is infinitesimally
polar at $x$.
\end{lem}

\begin{proof}
 Consider $T_x \mathcal F$ as the limit of rescaled singular Riemannian 
foliations $(M,\mathcal F)$ as in \cite{expl}, p.10.  As in \cite{expl},
we deduce that
$T_x \mathcal F$ is a singular Riemannian foliation on the Euclidean space
$T_x M$ such that at the regular point $v=\gamma ' (0) \in T_x M$ all
sectional curvatures vanish in local quotients. 
In this case, \pref{onepoint} below implies that $T_x \mathcal F$ is polar.
\end{proof}

\begin{prop} \label{onepoint}
Let $\mathcal F$ be a singular Riemannian foliation
on the Euclidean space $\R ^n$. Let $L$ be a regular leaf such that in
local quotients all sectional curvatures vanish at the image of this leaf.
Then $\mathcal F$ is polar.
\end{prop}

\begin{proof}
Since $\R ^n$ is flat, the sectional curvatures at the point $\{ L\} $ in local
projections vanish if and only if the O'Neill tensor $A :H_x \times H_x \to 
T_x (L(x))$ vanishes identically  at all points $x\in L$. But this implies
that each Bott-parallel normal field $H$ along $L$ is a parallel
normal field. Since all these fields are equifocal 
(cf. \cite{ATequifoc}), we get that $L$ 
is an isoparametric submanifold of $\R ^n$ and that $\mathcal F$
coincides with the isoparametric foliation defined by the isoparametric
submanifold $L$.  
\end{proof}

\section{Desingularization} \label{dessec}

\subsection{Notations}  First, let $T$ be a finite-dimensional real
vector space with scalar products $g$ and $g^+$. Let $A:T\to T$ be
the linear map defined by $g^+ (A(v),w) =g(v,w)$ for all $v,w \in T$.
Then, for each linear subspace $H$ of $T$, the image $A(H)$ of $H$
satisfies $H^{\perp _g} = (A(H)) ^{\perp _{g^+}}$, i.e., the $g$-orthogonal
complement of $H$ coincides with the $g^+$-orthogonal complement
of $A(H)$.   We will denote the map $A$ by $I_{g,g^+}$. By the same
symbol $I_{g,g^+}$ we denote the induced map on the Grassmannians 
$Gr_k (T)$, i.e., on the spaces of $k$-dimensional linear subspaces of $T$.
Note that $I_{g,g^+} \circ I_{g^+,g} =\Id$.    

 If $M$ is a Riemannian manifold with Riemannian metrics $g,g^+$ then 
we get a bundle automorphism $I_{g,g^+} :TM\to TM$ of the tangent 
bundle $TM$ of $M$.  For $k\geq  0$, we  denote
by $Gr_ k = Gr_k (M)$ the Grassmannian bundle  of the tangent bundle of $M$,
i.e., the bundle of $k$-dimensional subspaces of tangent spaces of $M$.
By the same symbol $I_{g,g^+}$ we will denote the induced bundle
automorphism $I_{g,g^+} :Gr_k  \to Gr_k $.  

 Let now $\mathcal F$ be a singular foliation adapted to the Riemannian
metrics $g$ and $g^+$, i.e., $\mathcal F$ is a singular Riemannian foliation
with respect to the Riemannian metrics $g$ and $g^+$.  For any point 
$x\in M$, we have the subspaces $H_x (g)$ and $H_x (g^+) $ of $g$-horizontal
and of $g^+$-horizontal vectors, respectively. By construction, our 
transformation $I_{g,g^+} $ satisfies $I_{g,g^+} (H_x(g))= H_x (g^+)$,
since $H_x$ is defined as orthogonal complement of the vertical space
$V_x$ that does not depend on the adapted Riemannian metric.

\subsection{Basic construction} \label{gauge}

Let $(M,g)$ be a Riemannian manifold and let $\mathcal F$ be an infinitesimally
polar singular Riemannian foliation on $M$ of codimension $k$.

  We denote by $\hat M \subset Gr_k$ the set 
of all $k$-dimensional infinitesimal sections of $\mathcal F$. Thus
$p^{-1} (x) \subset \hat M$ is the manifold of horizontal sections of
the polar Riemannian foliation $T_x \mathcal F$ on $T_x M$. In particular,
for each regular point $x\in M_0 \subset M$, the preimage  $p^{-1} (x)$
consists of only one point $H_x \in Gr_k M$.
 
We are going to prove:
\begin{enumerate}

\item $\hat M$ is a closed smooth submanifold of $Gr_k$.

\item The decomposition of $\hat M$ into preimages $\hat L = p^{-1} (L)$
of the leaves of  $\mathcal F$ is a smooth foliation $\hat {\mathcal F}$
of $\hat M$.
\end{enumerate}

The definition of $\hat M$ and 
of $\hat {\mathcal F}$ are local on $M$ and so are the claims.
 Thus we may restrict ourselves to a 
small distinguished neighborhood $U$ of a given point $x\in M$.
Pulling back the flat metric on $T_x M$ by the diffeomorphism $\phi$ (Subsection \ref{infinitsing}) , we thus 
reduce the question to  the following situation, to which we
will refer later as the {\it standard case}. The manifold
  $M$ is an open subset of the Euclidean
space $\mathbb R ^n$ with a flat (constant) Riemannian metric $g^+$; and
 $\mathcal F$  is the restriction of an  
isoparametric foliation on $\mathbb R^n$ to $M$.  Moreover, $g$ 
is a Riemannian  metric on $M$ adapted to $\mathcal F$.

Let $\hat M ^+$ be the subset of the Grassmannian 
$Gr _k$ of all infinitesimal horizontal
sections of $\mathcal F$ with respect to the  Riemannian metric
$g^+$.   Moreover, by $\hat {\mathcal F} ^+$ we denote the decomposition
of $\hat M^+$ into preimages of leaves of $\mathcal F$. Due to
\cite{Boualem},
 $\hat M^+$ is a closed submanifold of $Gr_k$ and 
$\hat {\mathcal F} ^+$ is a foliation on $\hat M^+$. (In fact, we only use
the result of Boualem in the case of an isoparametric foliation on the
flat $\R ^n$).

We claim that the  gauge $I_{g,g^+} :Gr_k  \to Gr_k $ sends $M$ to $\hat M ^+$.
As soon as the claim is verified, we deduce that 
$I_{g,g^+}$ sends $\hat {\mathcal F} $  to  $\hat {\mathcal F}^+$, because
$I_{g,g^+}$ is a bundle morphism, i.e., it commutes with the projection $p$.
Thus this claim would  imply that $\hat M$ is a smooth closed submanifold
and that  $\hat {\mathcal F}$ is a foliation on  $\hat M$.

 Thus it remains to prove the following

\begin{lem} 
Let $M$ be a manifold and let $\mathcal F$ be an infinitesimally polar
 singular Riemannian foliation
with respect to Riemannian metrics $g$ and $g^+$. 
Then $I_{g,g^+} :Gr_k \to Gr_k $ sends $\hat M$ to $\hat M^+$. 
\end{lem}

\begin{proof}
Choose a point $x\in M$. The singular foliation $T_x \mathcal F$ on the 
tangent space  $T_x M$ is defined independently of $g$ and $g^+$. The preimages
of $x$ in $\hat M$ and in $\hat M^+$ are defined only in terms of 
$T_x \mathcal F$, $g_x$ and $g_x ^+$, 
thus it is enough to prove the claim for the case 
$M=\mathbb R^n$, where $\mathcal F$ is a polar singular Riemannian foliation
 with respect to the flat
metrics $g$ and $g^+$ (by replacing $\mathcal F$ through $T_x \mathcal F$).
  In this case $\hat M$ and $\hat M^+$ are closed
submanifolds of $Gr_ k M$ and the regular part 
$p^{-1} (M_0 )$  is open and dense in both $\hat M$ and $\hat M ^+$
 (\cite{Boualem}).    By definition, $I_{g,g^+}$ sends 
$p^{-1} (M_0 ) \cap \hat M$ to $p^{-1} (M_0 ) \cap \hat M ^+$.

 By continuity, we deduce  $I_{g,g^+} (\hat M) \subset \hat M^+$.  Reversing 
the role of $g$ and $g^+$ and using that  $I_{g,g^+} \circ  I_{g^+,g} =Id$,
we deduce $I_{g,g^+} (\hat M) =  \hat M^+$. 
\end{proof}

\subsection{Regular vectors} Before we are going to define a Riemannian
structure on $\hat M$, we will need some observations concerning the space
of horizontal vectors. Let $\mathcal F$ be again a singular Riemannian
 foliation on a Riemannian manifold $(M,g)$. As in \cite{expl}, we denote by
 $D(g)$ 
the space of all unit horizontal vectors on $M$. By $D^0=D^0(g) \subset D(g)$
we denote the space of all {\it regular horizontal vectors}. Recall,
that a horizontal vector $v\in H_x$ is called regular if the horizontal
geodesic $\gamma ^v$ starting in the direction of $v$ contains at least
one regular point (\cite{expl}). 
In this case all but discretely many points on $\gamma ^v$
are regular.   Equivalently, one can say that a vector $v\in H_x$
is regular, if $v\in T_x M$ is a regular point of the infinitesimal
singular Riemannian foliation $T_x \mathcal F$.   Recall finally, that
$D^0 $ is a smooth, injectively immersed submanifold of the unit tangent
bundle $U^g M$ of $M$, that is invariant under the geodesic flow.

 If $\mathcal F$ is infinitesimally polar then a horizontal vector $v$ is 
regular  if and only if it is contained in only one horizontal section
$S$ of the isoparametric foliation $T_x \mathcal F$.  The assignment of the
section $S$ to the regular horizontal vector $v$ defines a map 
$m=m(g):D^0 \to \hat M$. We are going to prove that $m$ is a smooth submersion.

First, recall that for another Riemannian metric $g^+$ adapted to $\mathcal F$
we have an induced map $I_{g,g^+} :D (g)\to D (g^+)$ that is the restriction
of the smooth map $I_{g^+,g}$ between the unit tangent bundles 
$I_{g,g^+} : U^g M\to U^{g^+} M$ (induced by the fiber-wise linear isomorphisms
$ I_{g,g^+} : TM\to  TM$). 

\begin{lem}
 Let $\mathcal F$ be an infinitesimally polar singular Riemannian foliation
with respect to the Riemannian metrics $g$ and $g^+$. Then the map
$I_{g,g^+} :D (g)\to D (g^+)$ sends $D^0(g)$ to $D^0 (g^+)$.
\end{lem}

\begin{proof}
Since $I_{g,g^+}$ sends infinitesimal $g$-horizontal sections containing  a 
$g$-horizontal vector $v$ to infinitesimal $g^+$-horizontal sections
containing the $g^+$-horizontal vector $I_{g,g^+} (v)$, the result
follows from the characterization of $D^0$ as the set of all
horizontal vectors, contained in precisely one infinitesimal horizontal 
section. 
\end{proof}

\begin{quest}
 Is the statement  of the last lemma true for general singular Riemannian
foliations, that are not infinitesimally polar?
\end{quest}

 Let $M,\mathcal F, g, g^+$ be as in the lemma above, and let $\hat M$
and $\hat M ^+$ be the manifolds of horizontal infinitesimal sections
with respect to $g$ and $g^+$ respectively. We have the 
diffeomorphisms $I_{g,g^+} :D^0(g) \to D^0 (g^+)$ and 
$I_{g^+,g} :\hat M^+ \to \hat M$ and  the maps 
$m(g): D^0(g) \to \hat M$ and $m(g^+) :D^0 (g^+) \to \hat M^+$. 
By construction, the maps commute, i.e., $m(g) = 
I_{g^+,g} \circ m(g^+)\circ I_{g,g^+}$.  Therefore, $m(g)$ is a smooth
submersion if and only if $m(g^+)$ is a smooth submersion. Now we can prove:

\begin{lem}
Let $\mathcal F$ be an infinitesimal Riemannian foliation on a Riemannian
manifold $(M,g)$. Then the map $m(g) :D^0(g) \to \hat M $ is a smooth 
submersion. 
\end{lem}

\begin{proof}
 The objects $m(g),D^0,\hat M$ are defined locally on $M$. Thus it is enough
to prove the statement in a neighborhood of each point $x$ in $M$.  This
reduces the question to the {\it standard case}. 
Then the observation preceding
this proposition reduces the question to the case $\mathcal F =T_x \mathcal F$.
Thus we may assume that $M$ is the Euclidean space $\mathbb R ^n$ and
that $\mathcal F$ is a polar singular Riemannian foliation on $\mathbb R^n$. 

  In this case the claim can be deduced as follows. Given a regular horizontal 
vector $v\in D^0$, choose a small number $\epsilon$ and a neighborhood 
$O$ of $v$ in $D^0$ such that $p (\phi _{\epsilon }(O))$ is contained in the
set of regular points of $M$. Here, $p:UM\to M$ is the projection from the
unit tangent bundle to $M$ and $\phi _t$ is the geodesic flow. The 
Grassmannian bundle
of $\R^n$ is a trivial bundle with a canonical trivialization. With
respect to this trivialization we have  $m (v)= m(\phi _t (v))$ for all 
$v\in D^0$ and   all $t$.  Thus $m$ is preserved by  the geodesic
 flow $\phi$, and the above choice of $O$ reduces the question to the regular
part of $M$. However, in the regular part $M_0$ of $M$ the claim is clear.
\end{proof}

\subsection{Normal distribution}
 We are going to define now, what is
going to be the normal distribution of the foliation $\hat {\mathcal F}$
with respect to the Riemannian metric $\hat g$ to be defined later.
Let $\mathcal F$ be an infinitesimally
polar singular Riemannian foliation on a Riemannian manifold $M$. 
(Since we are 
not going to use auxiliary metrics $g^+$ anymore, we are going to suppress
the Riemannian metric $g$ in the sequel). Let $\hat M$ be defined as in
Subsection \ref{gauge}.  Let $M_0$ be the regular part of $M$ and let
$\hat M _0$ be the preimage $p^{-1} (M_0)$. The restriction
 $p:\hat M  _0 \to M_0$ is a diffeomorphism, thus on $\hat M _0$ there is
a smooth distribution $\hat {\mathcal H} _0 $ that is sent by $p$ to the
horizontal distribution of the Riemannian foliation $\mathcal F$ on
the Riemannian manifold $M_0$.  We claim:

\begin{lem} \label{distr}
There is a unique smooth $k$-dimensional distribution $\hat {\mathcal H}$
on $\hat M$ that extends $\hat {\mathcal H} _0$. 
\end{lem}

\begin{proof}
 The uniqueness is clear, since $\hat M_0$ is dense in $\hat M$. 
In order to prove the existence, it is enough to show that for each
element $S\in \hat M$ there are $k$ linearly independent smooth vector
fields $W_i$ defined on an open neighborhood  $O$ of $S$ in $\hat M$,
 such that the restriction of each $W_i$ to 
$O\cap \hat M_0$ is a section of $\hat {\mathcal H} _0$.

  Thus, let $S\in \hat M$ be given and let $x=p(S) \in M$ be the foot point
of $S$. Let $w \in T_x M$ be a regular unit horizontal vector contained
in $S$. Since the map $m:D^0 \to \hat M$ is a smooth submersion, we find
an open neighborhood $O$ of $S$ in $\hat M$ and a
 smooth section $n:O\to D^0$  with $m\circ n=\Id$ and $n(S)=w$.  

 Let $I$ be a small interval around $0$. Consider the map 
$\bar \xi :O\times I \to D^0$ given by 
$\bar \xi (\bar S, t) = \phi _t (n (\bar S))$, where $\phi _t$ denote
the restriction of the geodesic flow to $D^0$.  By construction,
$\bar \xi $ is a smooth map. This implies smoothness of the composition
$\xi  :O\times I  \to O$ given by $\xi  = m\circ \bar \xi$.
By construction, $\xi  (\bar S, 0) =\bar S$ for all $\bar S \in O$.
Therefore, the map $$W(\bar S) := \frac d {dt} \xi  (\bar S ,t)$$ 
is a smooth vector field on $O$.

 Now, the map $m:D^0 \to \hat M$ commutes with the projections to $M$,
i.e.,  $p(m(v))=p(v)$ for all $v\in D^0$. Thus the projection of any 
$\xi$-trajectory to $M$ is the projection of the corresponding $\bar \xi$
trajectory to $M$. By definition, $\bar \xi$-trajectories are flow lines
of the geodesic flow. Thus the  $\xi$-trajectory of 
a point $\bar S \in O$ is sent
by the projection $p:\hat M\to M$ to the regular  horizontal geodesic that
starts at $p$ in the direction $n (\bar S)$. In particular, we deduce
that the restriction of $W$ to $M_0$ is a section of $\hat H_0$. Moreover,
by construction, $p_{\ast} (W(\bar S))= w$.  
 
 Now, we choose a basis $w_i$ of $S$ that consists of regular vectors and 
applying the above construction, we get the linearly independent smooth
vector fields $W_i$, we were looking for.
\end{proof}

\subsection{Riemannian structure}  \label{riemsub}

Now we are in position to define 
the right  Riemannian structure $\hat g$ on $\hat M$.  We start 
with the canonical Riemannian metric $h$ on the Grassmannian bundle 
$Gr_k (M)$ (cf. \cite{tobenphd} for its definition and properties)
and denote by the same letter $h$ its restriction to the submanifold $\hat M$.
The projection $p: (Gr_k ,h)\to (M,g)$ is a Riemannian submersion.
In particular, the restriction $p:(\hat M ,h) \to (M,g)$ is $1$-Lipschitz.

 Let $\hat {\mathcal H}$ be the distribution of $k$-dimensional 
spaces on $\hat M$ defined in the previous subsection.  In the proof
of  \lref{distr} we have seen that for each $S\in \hat M$ it is possible
to choose a base $W_1,...,W_n$ of $\hat {\mathcal H} (S)$ that are mapped
by the differential $p_{\ast}$ to a base of $S \subset T_{p(S)} M$. In
particular, for each $S\in \hat M$, the restriction of 
$p_{\ast} :T_S \hat M \to T_{p(S)} M$ sends $\hat {\mathcal H} (S)$ bijectively
to $S\subset T_{p(S)} M$.  Since $S$ is normal
to the leaf of $\mathcal F$ through $p(S)$, we deduce that $\hat {\mathcal H}$
and $\hat {\mathcal F} $ are transversal.

 Now we define the Riemannian metric $\hat g$ on $\hat M$ uniquely
by the following three properties. On $\hat {\mathcal F}$ we let 
$\hat g$ coincide with the canonical metric $h$. We require 
$\hat {\mathcal F}$ and $\hat {\mathcal H}$ to be orthogonal with
respect to $\hat g$. Finally,
on $\hat {\mathcal H}$ we define $\hat g$ such that $p_{\ast}$ induces 
an isometry between $\hat {\mathcal H (S)}$ and $S$, for all elements 
$S\in \hat M$. In other words, we set 
$\hat g (v,w) = g(p_{\ast} (v), p_{\ast} (w))$, for all $v,w \in \hat 
{\mathcal H} (S)$.  

 By construction, $\hat g$ is a smooth Riemannian metric on $\hat M$. For
each point $S\in \hat M$, the differential $p_{\ast}$ sends the orthogonal
subspaces $\hat {\mathcal F} (S)$ and $\hat {\mathcal H} (S)$ to orthogonal
subspaces of $T_{p(S)} M$ and the restrictions of $p_{\ast}$ to
 $\hat {\mathcal F} (S)$ and to   $\hat {\mathcal H} (S)$ are $1$-Lipschitz.
Therefore, the map $p:(\hat M,\hat g) \to (M,g)$ is $1$-Lipschitz.  

On the regular part $\hat M _0$ the foliation $\hat {\mathcal F}$ is
a Riemannian foliation with respect to the
metric $\hat g$. (If $\hat M_0$ and $M_0$ are identified via the
diffeomorphism $p : \hat M_0 \to M_0$, the metric $\hat g$ arises from the
metric $g$ by changing $g$ only on $\mathcal F$ and by leaving the metric
on the normal part unchanged). Since $\hat M _0$ is dense in $\hat M$,
the foliation $\hat {\mathcal F}$ is a Riemannian foliation
on the whole manifold $(\hat M, \hat g)$.  

 By construction, $p_{\ast}$ sends horizontal vectors on $\hat M$ to
horizontal vectors on $M$ of the same length; therefore,
 $p$ preserves transverse
length of curves. Thus $p:(\hat M,\hat {\mathcal F})  \to (M,\mathcal F)$
is a geometric resolution.

\subsection{Proof of \tref{mainthm}}
Now we can finish the proof of \tref{mainthm}. If $(M,\mathcal F)$ admits a
geometric resolution, then $\mathcal F$ is infinitesimally polar, as was
shown in  Section \ref{onlyif}. 

Let now $\mathcal F$ be infinitesimally polar. Consider the manifold $\hat M$
with the foliation $\hat {\mathcal F}$ defined in Subsection \ref{gauge} and
let $F:\hat M\to M$ be the canonical projection $p$. Let $\hat g$ be the
Riemannian metric on $\hat M$ defined in Subsection \ref{riemsub}.  As we have
seen, $\hat {\mathcal F}$ is a Riemannian foliation on the Riemannian manifold
$(\hat M, \hat g)$ and $F:\hat M \to M$ is a geometric resolution.

 We have seen in Subsection \ref{riemsub}, that the map $F$ is $1$-Lipschitz.
By construction, the leaves of $\hat {\mathcal F}$ are
 preimages of leaves of $\mathcal F$, thus 
$p$ induces a bijection between spaces of leaves. Moreover, by construction,
the preimage of a compact subset $K$ on $M$ is a closed subset of a compact
subset of the Grassmannian bundle $Gr_k (M)$. Thus the map $F$ is proper.

 If $M$ is  compact then $\hat M$ is compact, since $F$ is proper.
Since $F$ is $1$-Lipschitz, a ball of radius $r$ around a point 
$S\in \hat M$ is contained in the preimage of the ball of radius $r$ around
$F(S)$ in $M$. If $M$ is complete, the properness
of $F$ implies that all balls in $\hat M$ are compact. Therefore,
 $\hat M$  is complete in this case.

  The objects ($\hat M, \hat {\mathcal F}, \hat g$)
are defined only in terms of $M,\mathcal F$ and $g$. Therefore, they are 
invariant under isometries of $(M,\mathcal F)$. This proves the statement about
$\Gamma$-equivariance. The claim about singular Riemannian foliations
$\mathcal F$ given by orbits of an isometric action of a group $G$ 
is a direct consequence of the last  claim.

 Assume now that $M$ and therefore $\hat M$ are complete. The notion 
of the absence of  horizontal conjugate point is a transverse notion , 
i.e., it can be formulated only in terms of local quotients (cf. \cite{expl}).
Since the transverse geometries of $(M,\mathcal F)$ and
of $(\hat M,\hat {\mathcal F})$ coincide, due to the definition of a geometric
resolution, the singular Riemannian foliation
$\mathcal F$ has no horizontal conjugate points if and only if the Riemannian
foliation 
$\hat {\mathcal F}$ has no horizontal conjugate points.

Identifying the regular part $\hat M_0$ with $M_0$ via $F$, we see that,
by construction, the horizontal distributions of $\mathcal F$ with respect to
the metrics $g$ and $\hat g$ coincide.  Thus, one of them is integrable if and
only if the other one is integrable. The integrability of the normal 
distribution on the regular part is equivalent to polarity 
(\cite{Alexandrino1}). This
shows that $\mathcal F$ is polar if and only if $\hat {\mathcal F}$ is polar.

This finishes the proof of \tref{mainthm}.

\section{Simplicity in the regular part} \label{toposec}

We are going to prove \tref{topothm} in a slightly more general
setting that we are going to describe now.

\begin{defn}
A singular Riemannian foliation on a Riemannian manifold $M$ is  {\bf full}
if for each leaf $L$ there is some $\epsilon >0$ such that $\exp (\epsilon v)$
is defined for each unit vector in the normal bundle  $L$.
\end{defn}

Each singular Riemannian foliation on a complete  
Riemannian manifold is full. In a full singular Riemannian foliation
 each pair of leaves is equidistant.  If $\mathcal F$ is full
on $M$ and if $U \subset M$
is an open subset  that is a union of leaves 
of $\mathcal F$ then the restriction of $\mathcal F$ to $U$ is   full again
(this follows from \cite{expl}, Proposition 4.3).
Moreover, for each covering $N$ of $M$ the lift of $\mathcal F$ to $N$
is full on $N$.

If $\mathcal F$ is a full singular Riemannian foliation
 on a Riemannian manifold $M$ with all leaves closed,
then $M/\mathcal F$ is a metric space, with  a natural inner metric
that has curvature locally bounded below in the sense of Alexandrov.
Note that an isometry of such a space is uniquely determined by its restriction
to an open subset.
Finally, a full regular Riemannian foliation 
is simple, i.e., has closed leaves with trivial holonomy, if and only if 
the quotient $M/\mathcal F$ is a Riemannian manifold.

 Let now $\mathcal F$ be a full singular Riemannian foliation
 on a connected Riemannian manifold $M$, with $\pi_1 (M) =\Gamma$.  Let
$\tilde M$ be the universal covering of $M$ and let $\tilde {\mathcal F}$ 
be the 
lifted singular Riemannian foliation on $\tilde M$. 
Assume that $\tilde {\mathcal F}$ has closed
leaves and denote by $B$ the quotient space  $\tilde M /\tilde {\mathcal F}$.
 The fundamental group   
$\Gamma$ acts on $(\tilde M,\tilde {\mathcal F})$. Thus we get
an induced action of $\Gamma$ on the quotient $B$.
 Denote by $\Gamma_0$ the kernel
of the action of $\Gamma$ on $B$, i.e., the set  of all elements of 
$\Gamma$ that act trivially on $B$.

\begin{lem} \label{simplelift}
In the notations above let $g\in \Gamma$ be an element. Then the following
are equivalent:
\begin{enumerate}
\item $g\in\Gamma _0$;

\item Each leaf $L$ of $\mathcal F$ contains a closed curve whose free 
homotopy class is the conjugacy class of $g$;

\item  There is a non-empty  open subset $U$ in $M$ such that each leaf $L$ of 
$\mathcal F$,  which has  a non-empty intersection with $U$, contains a
closed curve   whose free 
homotopy class is the conjugacy class of $g$.
\end{enumerate} 
\end{lem}

\begin{proof}
Let $\tilde L$ be a leaf of $\tilde {\mathcal F}$ 
through a point $y\in \tilde M$.
Then the translate  $gy$ is contained  in $\tilde L$ if and only if
$g$ fixes the point $\tilde L \in B$. On the other hand, if $gy$   is 
contained in $\tilde L$ then connecting $y$ and $gy$ by a curve in $\tilde L$
one obtains  a closed curve in the image  $L$ of $\tilde L$ in $M$
whose free homotopy class is in the conjugacy class of $g$. Note that
this image $L$ is a leaf of $\mathcal F$.

 Let  $L$ be  a leaf  in $M$ that contains a closed curve $\gamma$ whose
free homotopy class is in the conjugacy class of $g$. Then each lifted
leaf $\tilde L$ of $L$ contains a lift of the curve $\gamma$.  Thus, in this 
case, each lift $\tilde L$ of the leaf $L$ is fixed by some conjugate of $g$.

 Now the implications $1\Longrightarrow 2\Longrightarrow 3$ are clear. 
Assume $3$. Let $\tilde U$
be the preimage of $U$ in $\tilde M$ and $V$ the projection of $\tilde U$ to 
$B$.    Then $V$ is a non-empty  open subset of the quotient $B$ and each point
in $V$ is fixed by some conjugate of $g$.  There are only 
countably many conjugates of
$g$, each of them fixing a closed subset of $B$. By Baire's theorem, at least
one conjugate of $g$ fixes a non-empty open subset of $V$. 
 Since $B$ is an inner metric space
with  curvature locally bounded from below, $g$ fixes all of $B$. 
Therefore, $g\in \Gamma _0$.
\end{proof}

The following result  generalizes  \tref{topothm}.

\begin{prop}
Let $\mathcal F$ be a full singular Riemannian foliation on a simply
connected Riemannian  manifold $M$.  Let $M_0$ denote the regular stratum 
of $M$ and let the Riemannian foliation
$\mathcal F_0$ be the restriction of $\mathcal F$ to $M_0$. Assume that the lift
$\tilde {\mathcal F _0}$ of $\mathcal F_0$ to the universal covering $\tilde M_0$ is closed.
Then $\mathcal F_0$ is closed as well and the canonical projection $\tilde M_0 / \tilde {\mathcal F_0} \to M_0 /\mathcal F_0$
is an isometry. 
In particular, if $\tilde {\mathcal F_0}$ is a simple Riemannian foliation 
then $\mathcal F_0$ is a simple Riemannian foliation  on $M_0$.  
\end{prop}

\begin{proof}
 The assumptions and conclusions do not change if one deletes from $M$
all strata of codimension $\geq 3$.  Thus we may assume that
such strata do not exist.  
Then the complement $\Sigma  =M\setminus M_0$
is a disjoint union of closed submanifolds $\Sigma _i$ of codimension $2$.

 Choose a point $x_i$ on $\Sigma _i$, a small neighborhood $P_i$
of $x_i$ in $\Sigma _i$ and a small tubular neighborhood $U_i$
of $P_i$ in $M$.  Let $q:U_i \to P_i$ be the foot point projection.
The restriction of $q$ to $U_i \setminus P_i$ is a fiber bundle with circles
as fibers. By construction, each of these circles is contained
in a leaf of $\mathcal F$.

 On the other hand, all these circles are 
in the same free homotopy class $[g_i]$ of $U_i \setminus P_i$.  
Since $M$ is simply connected, the fundamental group $\Gamma$ of $M_0$ is 
generated by  conjugates of the elements $g_i$ (i.e., $\Gamma$ is normally
generated by the elements $g_i$).
Due to \lref{simplelift}, 
each of these free homotopy classes acts trivially on 
the Riemannian orbifold  $B=\tilde M_0 /\tilde {\mathcal F_0}$. Thus 
$\Gamma= \pi_1 (M_0)$ acts 
trivially on $B$ and we get $M_0 /\mathcal F_0 =B$.
This proves the theorem.
\end{proof}

 For a full Riemannian foliation $\mathcal F$ with closed leaves  one has an induced surjective homomorphism
 from $\pi _1 (M)$ onto $\pi_1 ^{orb} (B)$, the orbifold fundamental group of the quotient orbifold 
 $B$ (cf. \cite{Salem} or \cite{haefliger}).  Thus as a consequence of the above Proposition we deduce:
 
 \begin{cor} \label{pi1}
 Let $\mathcal F$ be a full singular Riemannian foliation on a simply connected Riemannian manifold $M$,
 with all leaves closed.
 Then the quotient $B_0$ of the restriction of $\mathcal F$ to the regular part $M_0$ is a Riemannian orbifold
 with $\pi _1 ^{orb} (B) =1$. 
 \end{cor}

\begin{rem}
The above \lref{simplelift}  is true also
in the case of non-closed $\tilde {\mathcal F}$, as one sees by localizing
the arguments.  \cref{pi1} is also valid without the assumption that $\tilde {\mathcal F_0}$ is closed,
in the sense, that the fundamental group of the pseudo-group of isometries $M_0 /\mathcal F_0$ is simply connected,
cf. \cite{Salem}.
\end{rem}  
  
  We  are going to use two simple observations about orbifolds. First of all, an orbifold $B$ with $\pi_1 ^{orb}(B)=1$ is
  orientable. Hence it does not have strata of codimension $1$, i.e., $\partial B =\emptyset$.  On the other
 hand, any non-compact 2-dimensional orbifold is a good orbifold. Thus if $B$ is a non-compact two-dimensional orbifold
 with $\pi_1 ^{orb} (B) =1$ then $B$ is a manifold (nessesarily an open disc).

 Now we are going to provide:
 \begin{proof}[Proof of \tref{boundary}]
  Let $p\in B$ be  a  point representing a singular leaf $L$ of $\mathcal F$. Choose some  $x\in L$.
  Choose a small distinguished neighborhood
$U$ at the point $x$. Then the restriction of $\mathcal F$ to $U$
is given by a (restriction of a) non-trivial  isoparametric foliation on $\R ^n$, thus
$U/\mathcal F$ is a Weyl chamber. The embedding $U\to M$ induces 
a finite-to-one projection $U/\mathcal F \to B$. Moreover, this projection
is given by a finite isometric action of a group $\Gamma$ on $U/\mathcal F$  
(cf. \cite{expl}, p.7). Since the Weyl chamber has non-empty boundary, so does 
its finite quotient. Hence any neighborhood of $p$ contains boundary points.
Since the boundary is closed, $p\in \partial B$.

 Assume now that $M$ is simply connected. Denote by the orbifold $B_0 \subset B$ the quotient of the regular part 
 of $\mathcal F$. We have seen  in \cref{pi1}, that $B_0$ is simply connected as orbifold. Thus it cannot contain strata of codimension $1$.
 But $B_0$ is open thus, if it  has a point in $\partial B$, then it has a point lying on a stratum of codimension $1$ in $B$.
 Then the whole stratum is contained in $B_0$, contradiction.
 \end{proof}

 Now it is easy to obtain:
 \begin{proof}[Proof of \cref{cor2d}]
 Recall  that  $\mathcal F$ is infinitesimally polar, since $B$ has dimension $2$ (\cite{expl}).
 Assume that the foliation is not regular. Then the quotient has non-empty boundary,
 by \tref{boundary}. Therefore, the complement of the boundary $B_0= B\setminus \partial B$ 
 (the quotient orbifold of the regular  part)  is not compact. But its orbifold
 fundamental group is trivial by \cref{pi1}. Since it is a 2-dimensional orbifold, it must be a manifold.  
 Thus there are no exceptional orbits.
 \end{proof}

    Now we are going to provide:
\begin{proof}[Proof of \tref{goodorb}]
 The equivalence of $(1)$, $(2)$ and $(4)$ has already been established (\tref{topothm} and \tref{boundary}).
 By definition of a Coxeter orbifold, $(3)$ implies $(4)$.

Now assume $(1)$. Take a point $p\in B$. As we have seen in the proof of  the first part of \tref{boundary},
there is a Weyl chamber $W$ with a Riemannian metric (a local quotient at a point $x\in M$ over $p$ and an action of a finite group $\Gamma$ on $W$ by isometries, such that the quotient $W/\Gamma$ is isometric to an open neighborhood of $p$.
Note that the set of regular points $W_0$  in $W$ is projected to $B_0$, the set of regular leaves. The assumption, that there are no exceptional orbits, i.e., that $B_0$ and therefore $W_0 /\Gamma$   is a Riemannian manifold, implies that the action of $\Gamma$ on the regular
part of $W$ is free. But the regular part $W_0$  of $W$ is contractible! Hence the finite group $\Gamma$ must act trivially on $W_0$
(since $\Gamma$ it has infinite cohomological dimension). Thus $\Gamma$ acts trivially on $W$. Therefore, a neighborhood of $p$
isometric to $W$. Thus $B$ is a Coxeter orbifold. 
\end{proof}

\end{document}